\documentclass{amsart}
\usepackage{amsmath, amsthm, amscd, amsfonts, amssymb}
\usepackage{tikz}
\usetikzlibrary{matrix,arrows}

\makeatletter \oddsidemargin.9375in \evensidemargin \oddsidemargin
\newtheorem{thm}{Theorem}[section]
\newtheorem{lem}[thm]{Lemma}
\newtheorem{pro}[thm]{Proposition}
\newtheorem{cor}[thm]{Corollary}
\theoremstyle{definition}
\newtheorem{den}[thm]{Definition}
\newtheorem{example}[thm]{Example}
\theoremstyle{remark}
\newtheorem{rem}[thm]{Remark}
\numberwithin{equation}{section}

\begin{document}

\title[Approximate cohomology] {Approximate  cohomology in Banach algebras}
\author[A. Shirinkalam and A. Pourabbas ]{A. Shirinkalam and  A. Pourabbas }
\address{Faculty of Mathematics and Computer
Science, Amirkabir University of Technology, 424 Hafez Avenue,
Tehran 15914, Iran}
\email{shirinkalam\_a@aut.ac.ir}

 \email{arpabbas@aut.ac.ir}



\begin{abstract}
 We introduce the notions of approximate cohomology and approximate homotopy in Banach algebras  and we study the relation between them.  We show that the approximate homotopically equivalent cochain complexes give the same approximate cohomologies. As an special case, approximate Hochschild cohomology is introduced and studied.

\end{abstract}

\maketitle

\section{Introduction and Preliminaries}
The notion of amenability for Banach algebras was first introduced by Johnson  \cite{J}. He  showed that the definition of amenability is equivalent to vanishing of $\mathcal{H}^n(\mathcal{A},X^*)$ for all $n$ and all $ \mathcal{A}$-module $X$. Later, Ghahramani and Loy \cite{GL} generalized the notion of amenability and they introduced the approximately amenable Banach algebras.  Here we introduce  an equivalent definition for approximate amenability. To do this first we generalize the definition of 
Hochschild cohomology to approximate cohomology and then we show that the definition of approximate  amenability is equivalent to vanishing of higher dimension approximate cohomology. 

The paper is organized as follows. After some preliminary material and definitions in  section 1,  we define a concept of approximate cohomology for Banach algebras in section 2. Also we define a notion of approximate homotopically equivalence relation for  two cochain complexes of Banach spaces and we show that if the cochains are approximate homotopically equivalent, then they will have the same approximate cohomologies.
In section 3, we apply these notions for a certain cochain complex  and we investigate some useful results.

     Let $ \mathcal{A} $ a Banach algebra and let $X$ be a Banach  $ \mathcal{A} $-bimodule.  The  dual space $ X^* $ of $  X$
becomes a Banach $\mathcal{A} $-bimodule via the following actions
\begin{eqnarray}\label{action}
	 \langle x, a\cdot \varphi \rangle =\langle x\cdot a, \varphi\rangle, \qquad \langle x, \varphi\cdot a \rangle =\langle a\cdot x, \varphi\rangle \qquad  (a\in \mathcal{A}, x\in X, \varphi \in X^*).
\end{eqnarray}
For a Banach algebra $ \mathcal{A} $, the projective tensor product $ \mathcal{A} \hat{\otimes}\mathcal{A}  $ is a Banach algebra with respect to the multiplication defined  by $ (a\otimes b)(c\otimes d)=ac \otimes bd $  \cite {BD}. If  $X$ and $Y$ are Banach  $  \mathcal{A}$-bimodules, then
$ X \hat{\otimes}Y $ becomes a Banach  $  \mathfrak{A}$-bimodule with the  canonical actions. Also, we use the isometric isomorphism $ (X \hat{\otimes}Y)^*\simeq \mathcal{BL}(X,Y^*)  $(\cite{BD}).

Recall that an $ \mathcal{A} $-bimodule $ X $ is called neo-unital if every $ x \in X $ can be written as $ a \cdot y \cdot b $ for some $a, b \in  \mathcal{A}  $ and $ y \in X $.
\begin{rem}\label{ess}
	Let $ \mathcal{A} $ be a Banach algebra with a bounded approximate identity and let $ X $ be a Banach $ \mathcal{A} $-bimodule. Then by \cite[Proposition 1.8]{J}, the subspace $ X_{ess} := \mathrm{lin} \{a\cdot x\cdot b : a, b \in  \mathcal{A},\, x \in X   \}
	$ is a  neo-unital closed sub-$ \mathcal{A} $-bimodule of $ X $. Moreover, $ X_{ess}^{\perp} $ is complemented in $ X^* $. \end{rem}

\begin{den}\cite[Definition 5.2.1]{R}
	A cochain complex  $ E=(E_n,\delta^n)_{n\in \mathbb{Z}} $ of Banach spaces is a  sequence of Banach spaces along with bounded linear maps  $$\cdot \cdot \cdot \xrightarrow{\delta^{n-1}}E_n \xrightarrow{\delta^{n}}E_{n+1}\xrightarrow{\delta^{n+1}}E_{n+2}\xrightarrow{\delta^{n+2}}\cdot \cdot \cdot  $$
	such that for every $ n\in \mathbb{Z} $, $\delta^{n}\circ \delta^{n-1}=0.  $
\end{den}
\begin{den}\cite[Definition 5.2.3]{R}
	Let $ E=(E_n,\delta^n)_{n\in \mathbb{Z}} $ be a cochain complex of Banach spaces.
	Let $ \mathcal{Z}^n(E):=\ker \delta^n $ and let $ \mathcal{B}^n(E):=\hbox{ran}\, \delta^{n-1 }$.
	Then
	$$ \mathcal{H}^n(E):=\mathcal{Z}^n(E)/\mathcal{B}^n(E), $$

	is called  the $ n- $th  cohomology group of the cochain complex $ E. $
\end{den}
	Let $ \mathcal{A} $ be a Banach algebra and let $ X $ be  a Banach $ \mathcal{A} $-bimodule. Then, for every $ n \in \mathbb{N} $, the space of all bounded, $ n- $linear maps  from $ \mathcal{A} $ into $ X $
	is denoted by
	$ \mathcal{BL}^n(\mathcal{A},X) $.
	 \begin{example}\label{Hoc}
	Let $ \mathcal{A} $ be a Banach algebra and let $ X $ be  a Banach $ \mathcal{A} $-bimodule.  For every $ n\geq 1$, we take $ E_n=\mathcal{BL}^n(\mathcal{A},X)$.
	Let $ \delta^n:\mathcal{BL}^n(\mathcal{A},X)\rightarrow \mathcal{BL}^{n+1}(\mathcal{A},X) $ be defined by
	\begin{align*}
		(\delta^nT)(a_1,..., a_{n+1})
		&:=a_1\cdot T(a_2,..., a_{n+1})\\
		&+\sum_{k=1}^{n}(-1)^kT(a_1,..., a_ka_{k+1},..., a_{n+1})\\
		&+(-1)^{n+1}T(a_1,..., a_{n})\cdot a_{n+1}, \quad (T\in E_n, a_1,..., a_{n+1}\in \mathcal{A}).
	\end{align*}	
	In the case $n=0$ we take $ E_0=X $ and $\delta^0:X\to\mathcal{BL}^1(\mathcal{A},X)$ is given by $\delta^0(x)(a)=
	a\cdot x-x\cdot a.$
	Then the sequence $ E=(E_n,\delta^n)_{n\geq 0} $ is called the Hochschild cochain complex and $$\mathcal{H}^n(\mathcal{A},X):=\mathcal{Z}^n(\mathcal{A},X)/\mathcal{B}^n(\mathcal{A},X),\quad\hbox{and}\quad \mathcal{H}^0(\mathcal{A},X)=\{x\in X:a\cdot x=x \cdot a,\forall a\in\mathcal{A}\}$$ is called the $ n$-th Hochschild cohomology of 	$\mathcal{A}$
	with coefficients in $X$.
		\end{example}
\begin{den}\cite[Definition 5.2.5]{R}
Let $ E=(E_n,\delta^n_E)_{n\in \mathbb{Z}} $	 and $ F=(F_n,\delta^n_F)_{n\in \mathbb{Z}} $ be two cochain complexes of Banach spaces. A morphism $ \varphi : E\rightarrow F $ is a family $ (\varphi_n)_{n\in \mathbb{Z}} $ of continuous linear maps $ \varphi_n:E_n \rightarrow  F_n $ such that for every $n\in \mathbb{Z}  $, the following diagram commutes,
\begin{center}
	\begin{tikzpicture}
	\matrix [matrix of math nodes,row sep=1cm,column sep=1cm,minimum width=1cm]
	{
		|(A)| \displaystyle E_n  &   |(B)|  E_{n+1}    \\
		|(C)|    F_n     &   |(D)|  F_{n+1}. \\
	};
	\draw[->]    (A)-- node [above] { $ \delta^n_E $}(B);
	\draw[->]   (A)--  node [left] { $ \varphi_n $} (C);
	\draw[->]  (C)-- node [below]   { $\delta^n_F $}(D);
	\draw[->]  (B)-- node [right]  {$ \varphi_{n+1} $}(D);
	\end{tikzpicture}
\end{center}
\end{den}
\section{Approximate cohomology }
In  Example \ref{Hoc}, when we consider $n=1$  the elements of
$\mathcal{Z}^1(\mathcal{A},X)$ are called continuous dervations and the elements of $\mathcal{B}^1(\mathcal{A},X)$ are called the inner derivations.

A continuous derivation $D:\mathcal{A}\longrightarrow X$ is
approximately inner if there is a  net
$(x_{\nu})\subset X$
such that for every $a\in A$
$$D(a)=\lim_{\nu}\delta^0(x_{\nu})(a)=\lim_{\nu}
a\cdot x_{\nu}-x_{\nu}\cdot a.$$ the limit  being in norm. That is, $D=\hbox{st}-\lim_{\nu}\delta^0(x_{\nu})$, the limit is taken in the strong  topology of
$\mathcal{B}(\mathcal{A},X)$.
A Banach algebra $ \mathcal{A} $ is called approximately amenable, if for every Banach $ \mathcal{A}$-bimodule $ X $,  every continuous  derivation $  D:\mathcal{A}\rightarrow X^* $ is approximately  inner (\cite[Definition 1.2]{GL}).
 This definition  was a motivation for  this section to introduce a  concept of approximate cohomology and  approximate homotopy. We show that if two complexes are  approximate homotopically equivalent, then they give the same approximate cohomologies.
\begin{den}
 Let $ (E_n)_{n\in \mathbb{Z}} $ be a sequence of Banach spaces such that each $ E_n $ is equipped with a topology $ \tau_n. $ For every $ n\in \mathbb{Z} $, suppose that $ \delta^n:(E_n, \tau_n) \rightarrow (E_{n+1},\tau_{n+1}) $ is a continuous linear map  and let $ \delta^{n}\circ \delta^{n-1}=0. $ Then the triple $ (E,\tau,\delta) $ is called a $ \tau $-cochain complex, where $ E=(E_n)_{n\in \mathbb{Z}}$, $ \tau=(\tau_n)_{n\in \mathbb{Z}}  $ and $ \delta=(\delta^n)_{n\in \mathbb{Z}}  $.
\end{den}
Sometimes for simplicity, we use  $ E=(E_n,\tau_n, \delta^n)_{n\in \mathbb{Z}}  $ for a $ \tau $-cochain complex.
\begin{den}
	Let $ E=(E_n,\tau_n, \delta^n)_{n\in \mathbb{Z}} $ be a $ \tau $-cochain complex of Banach spaces. For every $ n\in \mathbb{Z} $,
	let  $ \mathcal{Z}^n(E):=\ker \delta^n$ and $ \mathcal{B}^n(E):=\hbox{ran}\, \delta^{n-1 }$.
	Then $ \mathcal{Z}^n(E) $ is a $ \tau_n$-closed subspace of $E_n  $ and
	 $ \mathcal{B}^n(E)\subseteq \mathcal{Z}^n(E) $ is a subspace. We define
 $$\mathcal{H}^n_{app}(E):=\mathcal{Z}^n(E)/{\overline{\mathcal{B}^n(E)}}^\tau,$$
	 where the closure is taken in the  topology $ \tau_n $ of $ E_n $.
	 We call this quotient space, the $ n$-th \textit{approximate cohomology} of the complex $ E. $
\end{den}
\begin{pro}\label{comp}
	Let $ E=(E_n,\tau_n, \delta^n)_{n\in \mathbb{Z}} $ be a $ \tau $-cochain complex of Banach spaces. If $  \mathcal{H}^n(E)=\{ 0\} $, then $  \mathcal{H}^n_{app}(E)=\{ 0\}. $
	\end{pro}
\begin{proof}
	Consider a surjective map $ \theta :\mathcal{Z}^n(E) \rightarrow \mathcal{H}^n_{app}(E) $ defined by $$ \theta(T)= T+\overline{\mathcal{B}^n(E)}^{\tau_n}\quad ( T\in \mathcal{Z}^n(E)). $$
	The kernel of $ \theta $ contains $\mathcal{B}^n(E)  $, so there is an induced surjective map $ \tilde{\theta}:\mathcal{H}^n(E)\rightarrow  \mathcal{H}^n_{app}(E) $ defined by $\tilde{\theta}(T+\mathcal{B}^n(E))=\theta (T)  $. Clearly $\tilde{\theta}  $ is injective if and only if $\mathcal{B}^n(E)  $
	is $ \tau_n $-closed in $ \mathcal{Z}^n(E) $. Hence,  $  \mathcal{H}^n(E)=\{ 0\} $ implies that  $  \mathcal{H}^n_{app}(E)=\{ 0\}. $
\end{proof}
Note that the converse of Proposition \ref{comp} does not hold in general (see Example \ref{exa}  below).
\begin{den}\label{d1}
Let $ E=(E_n,\tau_n,\delta^n_E)_{n\in \mathbb{Z}} $	 and $ F=(F_n,\rho_n,\delta^n_F)_{n\in \mathbb{Z}} $ be  $ \tau $-cochain  and $ \rho $-cochain complexes, respectively. A \textit{$ (\tau,\rho) $-morphism} $ \varphi : E\rightarrow F $ is a morphism $ \varphi=(\varphi_n)_{n\in \mathbb{Z}} $ such that each $ \varphi_n:(E_n, \tau_n)\rightarrow (F_n, \rho_n) $ is continuous.

\end{den}
The proof of the following lemma is straightforward.
\begin{lem}\label{1}
Let $ E=(E_n,\tau_n,\delta^n_E)_{n\in \mathbb{Z}} $	 and $ F=(F_n,\rho_n,\delta^n_F)_{n\in \mathbb{Z}} $ be  $ \tau $-cochain  and $ \rho $-cochain complexes and let $ \varphi : E\rightarrow F $ be  a $ (\tau,\rho) $-morphism.
 For every $ n $, we have $\varphi_n(\mathcal{Z}^n(E))\subseteq \mathcal{Z}^n(F)  $, and
	$\varphi_n(\overline{\mathcal{B}^n(E)}^{\tau_n})\subseteq \overline{\mathcal{B}^n(F)}^{\rho_n}  $.
\end{lem}
 \begin{cor}\label{bar}
 	Let $ E $, $ F$ and $ \varphi$ be as in Definition \ref{d1}. Then, $ \varphi $ induces a sequence $ \overline{\varphi}=(\overline{\varphi_n})_{n\in \mathbb{Z}} $ of linear maps
 	$$\overline{\varphi_n}: \mathcal{H}^n_{app}(E)\rightarrow \mathcal{H}^n_{app}(F). $$
  	\end{cor}
\begin{proof}
	By Lemma \ref{1}, $\varphi_n(\mathcal{Z}^n(E))\subseteq \mathcal{Z}^n(F)  $ and
	$\varphi_n(\overline{\mathcal{B}^n(E)}^{\tau_n})\subseteq \overline{\mathcal{B}^n(F)}^{\rho_n}  $. Thus we can define a map $ \overline{\varphi_n}: \mathcal{H}^n_{app}(E)\rightarrow \mathcal{H}^n_{app}(F) $  by $\overline{\varphi_n}(T+\overline{\mathcal{B}^n(E)}^{\tau_n}  )= \varphi_n(T)+ \overline{\mathcal{B}^n(F)}^{\rho_n} \quad (T\in \mathcal{Z}^n(E))$,
	as required.
\end{proof}
\begin{den}
Let $ E=(E_n,\tau_n,\delta^n_E)_{n\in \mathbb{Z}} $	 and $ F=(F_n,\rho_n,\delta^n_F)_{n\in \mathbb{Z}} $ be  $ \tau $-cochain  and $ \rho $-cochain complexes, respectively. The $ (\tau,\rho)$-morphisms $ \varphi= (\varphi_n)_{n\in \mathbb{Z}}  $ and $ \psi =  (\psi_n)_{n\in \mathbb{Z}}  $	from $ E $ into $ F $
	are called \textit{approximately homotopic}, if  for every $n\in \mathbb{Z}$ there exists a family of bounded linear maps $  (\xi^n_\alpha)_{\alpha} $ from $ E_{n+1}$ into $F_n$
	such that for every $ T\in E_n $,
	  we have
	$$(\varphi_n - \psi_n)(T)=\rho_n-\lim_\alpha \,(\delta_F^{n-1}\circ \xi_\alpha^{n-1}+\xi^n_\alpha\circ \delta^n_E)(T),  $$
	whenever the limit exists.
	The family $ (\xi^n_\alpha)_{\alpha, n} $ is called an approximate homotopy of $ \varphi $ and $ \psi $.
\end{den}
\begin{pro}
Let $ E=(E_n,\tau_n,\delta^n_E)_{n\in \mathbb{Z}} $	 and $ F=(F_n,\rho_n,\delta^n_F)_{n\in \mathbb{Z}} $ be  $ \tau $-cochain  and $ \rho $-cochain complexes, respectively. Let $ \varphi, \psi :E \rightarrow F $
	be two approximately homotopic $(\tau,\rho)$-morphisms. Then for every $n\in \mathbb{Z}$, we have $\overline{\varphi_n}=\overline{\psi_n}$, where $ \overline{\varphi_n},\overline{\psi_n}:\mathcal{H}^n_{app}(E)\rightarrow \mathcal{H}^n_{app}(F) $ are the induced maps as in Corollary \ref{bar}.
\end{pro}
\begin{proof}
	 By assumption,
there exists a family of bounded linear maps $  (\xi^n_\alpha)_{\alpha, n} $ such that
	$$(\varphi_n - \psi_n)(T)=\rho_n-\lim_\alpha \,(\delta_F^{n-1}\circ \xi_\alpha^{n-1}+\xi^n_\alpha\circ \delta^n_E)(T), $$
where $ T\in \mathcal{Z}^n(E)$. Thus $\varphi_n(T) - \psi_n(T) = \rho_n-\lim_\alpha \,\delta_F^{n-1}\circ \xi_\alpha^{n-1}(T) $,
	so $\varphi_n(T) - \psi_n(T)\in \overline{\mathcal{B}^n(F)}^{\rho_n}.  $ Hence
	$\overline{\varphi_n}(T)=\varphi_n(T) + \overline{\mathcal{B}^n(F)}^{\rho_n}=
	\psi_n(T) + \overline{\mathcal{B}^n(F)}^{\rho_n}=\overline{\psi_n}(T).$
\end{proof}
\begin{den}
	Suppose that  $ E=(E_n,\tau_n,\delta^n_E)_{n\in \mathbb{Z}} $	 and $ F=(F_n,\rho_n,\delta^n_F)_{n\in \mathbb{Z}} $ are  $ \tau $-cochain  and $ \rho $-cochain complexes, respectively.
Then $ E $ and $F$	 are called \textit{approximate homotopically equivalent}, if there are two $ (\tau,\rho)$-morphism  $ \varphi:E \rightarrow F $ and $ (\rho,\tau)$-morphism $ \psi: F \rightarrow E $ such that $ \varphi\circ\psi $ and $ \psi\circ\varphi $
	are approximately homotopic to the identity morphism on $ F $ and $ E $, respectively, that is,
 there exist two family of bounded linear maps
	$  (\xi^n_\alpha)_{\alpha, n} $ and $  (\eta^n_\beta)_{\beta, n}$
	such that for every $n\in \mathbb{Z}  $  we have
	$\xi^n_\alpha: E_{n+1}  \rightarrow E_n $ and $ \eta^n_\beta: F_{n+1}  \rightarrow F_n $ satisfying
			\begin{eqnarray}\label{111}
		(\psi_n\circ \varphi_n )(T)-T=\tau_n-\lim_\alpha \,(\delta^{n-1}_E\circ \xi_\alpha^{n-1}+\xi^n_\alpha\circ \delta^n_E)(T), \quad (T\in E_n )
	\end{eqnarray}
	and
	\begin{eqnarray}\label{222}
		(\varphi_n\circ \psi_n )(S)-S=\rho_n-\lim_\beta \,(\delta^{n-1}_F\circ \eta_\beta^{n-1}+\eta^n_\beta\circ \delta^{n}_F)(S), \quad (S\in F_n ).
	\end{eqnarray}
		\end{den}
\begin{thm}
Let $ E=(E_n,\tau_n,\delta^n_E)_{n\in \mathbb{Z}} $	 and $ F=(F_n,\rho_n,\delta^n_F)_{n\in \mathbb{Z}} $ be  approximate homotopically equivalent. Then for every $n\in \mathbb{Z}  $,
	$$ \mathcal{H}^n_{app}(E)\cong \mathcal{H}^n_{app}(F). $$
\end{thm}
\begin{proof}
	Since $ E $ and $F  $ are approximate homotopically equivalent, there exist
	$ (\tau,\rho)$-morphism  $ \varphi:E \rightarrow F $ and $ (\rho,\tau)$-morphism $ \psi: F \rightarrow E $ and two
	 family of continuous linear maps
	$  (\xi^n_\alpha)_{\alpha, n} $ and $  (\eta^n_\beta)_{\beta, n}$
	which satisfy in the equations (\ref{111}) and (\ref{222}).

Let $n\in \mathbb{Z}  $	be fixed. Consider the induced map $ \overline{\varphi_n}:\mathcal{H}^n_{app}(E)\rightarrow \mathcal{H}^n_{app}(F) $
which is	defined  in  Corollary \ref{bar}. We show that $ \overline{\varphi_n} $ is an isomorphism.
	Let $ T\in \mathcal{Z}^n(E) $ be such that $  \overline{\varphi_n}(T+\overline{\mathcal{B}^n(E)}^{\tau_n})=0 $. Then $\varphi_n(T) \in \overline{\mathcal{B}^n(F)}^{\rho_n}  $, so that $\psi_n \circ \varphi_n(T) \in \overline{\mathcal{B}^n(E)}^{\tau_n} $. Now by (\ref{111}) we have
 $$\psi_n \circ \varphi_n(T)-T= \tau_n-\lim_\alpha \,\delta^{n-1}_E\circ \xi_\alpha^{n-1}(T) \in \overline{\mathcal{B}^n(E)}^{\tau_n}. $$
	Thus $ T\in \overline{\mathcal{B}^n(E)}^{\tau_n}. $ This means that $\overline{\varphi_n}  $ is one-to-one.
	
	Now, we show that $ \overline{\varphi_n} $
	is surjective. Let $ S\in \mathcal{Z}^n(F) $ be such that $S+ \overline{\mathcal{B}^n(F)}^{\rho_n}\in \mathcal{H}^n_{app}(F)$. By Lemma \ref{1},
	$ \psi_n(S)\in \mathcal{Z}^n(E) $, so $\psi_n(S)+\overline{\mathcal{B}^n(E)}^{\tau_n}\in \mathcal{H}^n_{app}(E).  $
	We show that
	\begin{eqnarray}\label{00}
		\overline{\varphi_n}(\psi_n(S)+\overline{\mathcal{B}^n(E)}^{\tau_n}) = S+\overline{\mathcal{B}^n(F)}^{\rho_n}.
	\end{eqnarray}
Using (\ref{222}) we obtain $$ \varphi_n\circ\psi_n(S)-S= \rho_n-\lim_\beta \,\delta^{n-1}_F\circ \eta_\beta^{n-1}(S) \in \overline{\mathcal{B}^n(F)}^{\rho_n}. $$
	  Hence $	 \varphi_n\circ\psi_n(S)+\overline{\mathcal{B}^n(F)}^{\rho_n} = S+\overline{\mathcal{B}^n(F)}^{\rho_n}  $ which shows that the equation (\ref{00}) holds.
	  Thus $ \overline{\varphi_n} $ is an isomorphism.
\end{proof}
\section{Approximate Hochschild cohomology}
	Let $ \mathcal{A} $ be a Banach algebra and let $ X $ be  a Banach $ \mathcal{A} $-bimodule.
In this section, we study the approximate cohomology in the special case when $ E=(E_n,\tau_n,\delta^n)_{n\geq 0} $ is the Hochschild cochain complex and $ \tau_n $ is the strong topology on $ E_n = \mathcal{BL}^n(\mathcal{A},X)$.
By definition, a net $ (T_i)_i \subseteq E_n $ converges strongly to $ T\in E_n $ if
 $ \Arrowvert (T_i-T)(a_1,..., a_{n})\Arrowvert \rightarrow 0 \quad ( a_1,..., a_{n} \in \mathcal{A})$.

Note that for every  $n\in \mathbb{N}$, the map $\delta^n: \mathcal{BL}^n(\mathcal{A},X)\rightarrow \mathcal{BL}^{n+1}(\mathcal{A},X) $ is  strongly continuous. Throughout this section, for simplicity, we  denote the Hochschild cochain complex by  $ E=(E_n,\delta^n)_{n\geq 0} $ and all the limits being taken in the strong topology.
\begin{den} Let $ \mathcal{A} $ be a Banach algebra and let $ X $ be  a Banach $ \mathcal{A} $-bimodule. The space $\ker \delta^n $ is denoted by
$ \mathcal{Z}^n(\mathcal{A},X)$. The elements of $\mathcal{Z}^n(\mathcal{A},X) $ are called the $n$-cocycles and the elements of $\hbox{ran}\, \delta^{n-1 }=\mathcal{B}^n(\mathcal{A},X)$ are called the $n$-coboundaries. Then for $n\geq 1$ we define
		$$\mathcal{H}^n_{app}(E)= \mathcal{H}^n_{app}(\mathcal{A},X) :=\mathcal{Z}^n(\mathcal{A},X)/\overline{\mathcal{B}^n(\mathcal{A},X)}^{\text{ strong}}, $$
	is called the $ n$-th \textit{approximate Hochschild cohomology} of the complex $ E $.
\end{den}
\begin{example}
	Let $ \mathcal{A} $ be a Banach algebra and let $ X $ be  a Banach $ \mathcal{A} $-bimodule.	By definition of approximate amenability, $  \mathcal{A} $ is approximately amenable if and only if every continuous derivation $  D:\mathcal{A}\rightarrow X^* $ is a strong-limit of inner derivations  which means that $ \mathcal{H}^1_{app}(\mathcal{A} ,X^*)=\{ 0\}. $
	\end{example}
\begin{example}	\label{exa}
Gahramani and Loy \cite{GL} constructed an approximate amenable Banach algebra which is not amenable. This example shows that the notions of approximate Hochschild cohomology and  Hochschild cohomology are distinct, at least for  $ n=1 $.
	\end{example}

\begin{rem}\label{triv}
	Let $ \mathcal{A} $ be a Banach algebra with a bounded approximate identity and let $ X $ be  a Banach $ \mathcal{A} $-bimodule. If $ \mathcal{A} $ acts trivially from one side, then by \cite[Proposition 1.5]{J}, for every $ n $, $ \mathcal{H}^n(\mathcal{A},X^*)=\{ 0\} $. So by Proposition \ref{comp}, $ \mathcal{H}^n_{app}(\mathcal{A},X^*)=\{ 0\} $.
	\end{rem}
\begin{thm}\label{long}
	Let $ \mathcal{A} $ be a Banach algebra and let $ X $ be  a Banach $ \mathcal{A} $-bimodule. Let $ Y $ be a closed submodule of $ X $ which is complemented as a Banach space. Then, there are maps making the following sequence exact.
	$$\cdot \cdot \cdot  \rightarrow \mathcal{H}^n_{app}(\mathcal{A},Y) \rightarrow \mathcal{H}^n_{app}(\mathcal{A},X)\rightarrow \mathcal{H}^n_{app}(\mathcal{A},X/Y)\rightarrow \mathcal{H}^{n+1}_{app}(\mathcal{A},Y) \rightarrow  \cdot \cdot \cdot$$
\end{thm}
\begin{proof}
We denote by $ \iota $, the injection map $ Y \hookrightarrow X$ and  by $ q $, the usual quotient map $ X \twoheadrightarrow X/Y. $
Since $Y$ is complemented, $q$ is an isomorphism, so it has a bounded right inverse $ p:X/Y \rightarrow X. $

We define a map $ \iota_0:\mathcal{BL}^{n}(\mathcal{A},Y)\rightarrow \mathcal{BL}^{n}(\mathcal{A},X) $  by $ \iota_0(T)=\iota\circ T $ and a map $ \iota_1:\mathcal{BL}^{n+1}(\mathcal{A},Y)\rightarrow \mathcal{BL}^{n+1}(\mathcal{A},X) $  by $ \iota_1(T)=\iota\circ T $. Also $ q_0:\mathcal{BL}^{n}(\mathcal{A},X)\rightarrow \mathcal{BL}^{n}(\mathcal{A},X/Y) $ is defined  by $ q_0(T)=q\circ T $ and $ j:\mathcal{BL}^{n}(\mathcal{A},X/Y)\rightarrow \mathcal{BL}^{n+1}(\mathcal{A},Y) $ is defined by $ j(T)=\delta^n_XpT-p\delta^n_{X/Y}T $, where $ \delta^n_X:\mathcal{BL}^{n}(\mathcal{A},X)\rightarrow \mathcal{BL}^{n+1}(\mathcal{A},X) $ and $\delta^n_{X/Y}:\mathcal{BL}^{n}(\mathcal{A},X/Y)\rightarrow \mathcal{BL}^{n+1}(\mathcal{A},X/Y)  $ are the corresponding coboundary operators.
Note that since $$ qj(T)=q\delta^n_XpT-qp\delta^n_{X/Y}T=\delta^n_{X/Y}qpT-\delta^n_{X/Y}T=0, $$
we have $ jT\in \mathcal{BL}^{n+1}(\mathcal{A},Y) $ for all $T\in\mathcal{BL}^{n}(\mathcal{A},X/Y)$.

If $ T\in \mathcal{Z}^{n}(\mathcal{A},Y) $, then $ \iota_0(T)\in \mathcal{Z}^{n}(\mathcal{A},X) $, thus $\iota_0  $ maps $ \mathcal{Z}^{n}(\mathcal{A},Y) $ into $ \mathcal{Z}^{n}(\mathcal{A},X) $. Also, an easy verification shows that $ q_0 $ maps $ \mathcal{Z}^{n}(\mathcal{A},X) $ into $ \mathcal{Z}^{n}(\mathcal{A},X/Y) $ and $ j $ maps $ \mathcal{Z}^{n}(\mathcal{A},X/Y) $ into $ \mathcal{Z}^{n+1}(\mathcal{A},Y) $. Thus we obtain the sequence
\begin{eqnarray}\label{seq1}
		\cdot \cdot \cdot \rightarrow\mathcal{Z}^{n}(\mathcal{A},Y)\xrightarrow{\iota_0} \mathcal{Z}^{n}(\mathcal{A},X) \xrightarrow{q_0} \mathcal{Z}^{n}(\mathcal{A},X/Y) \xrightarrow{j} \mathcal{Z}^{n+1}(\mathcal{A},Y)\xrightarrow{\iota_1} \cdot \cdot \cdot .
\end{eqnarray}

If $ T\in \mathcal{Z}^{n-1}(\mathcal{A},X/Y) $, then
$$j \delta^{n-1}_{X/Y}T=\delta^n_{X}p\delta^{n-1}_{X/Y}T-p\delta^{n}_{X/Y}\delta^{n-1}_{X/Y}T=\delta^n_{X}p\delta^{n-1}_{X/Y}T=-\delta^n_{Y}(\delta^{n}_{X}pT-p\delta^{n}_{X/Y}T)=\delta^n_{Y}(-jT). $$
Thus $ j $ maps $ \mathcal{B}^{n}(\mathcal{A},X/Y) $ into $\mathcal{B}^{n+1}(\mathcal{A},Y)$.   Similar arguments show that $\iota_0  $ maps $ \mathcal{B}^{n}(\mathcal{A},Y) $ into $ \mathcal{B}^{n}(\mathcal{A},X) $ and $ q_0 $ maps $ \mathcal{B}^{n}(\mathcal{A},X) $ into $ \mathcal{B}^{n}(\mathcal{A},X/Y) $ and we have the sequence
\begin{eqnarray}\label{seq}
	 \cdot \cdot \cdot\rightarrow \mathcal{B}^{n}(\mathcal{A},Y)\xrightarrow{\iota_0} \mathcal{B}^{n}(\mathcal{A},X) \xrightarrow{q_0} \mathcal{B}^{n}(\mathcal{A},X/Y) \xrightarrow{j} \mathcal{B}^{n+1}(\mathcal{A},Y)\xrightarrow{\iota_1}  \cdot \cdot \cdot.
\end{eqnarray}
Since $ \iota_0, q_0 $ and $ j $ are strongly continuous, the sequence (\ref{seq}) extends to the sequence
$$
	 \cdot \cdot \cdot \rightarrow\overline{\mathcal{B}^{n}(\mathcal{A},Y)}^{\text{ strong}}\xrightarrow{\iota_0} \overline{\mathcal{B}^{n}(\mathcal{A},X)}^{\text{ strong}} \xrightarrow{q_0} \overline{\mathcal{B}^{n}(\mathcal{A},X/Y)}^{\text{ strong}} \xrightarrow{j} \overline{\mathcal{B}^{n+1}(\mathcal{A},Y)}^{\text{ strong}}\xrightarrow{\iota_1} \cdot \cdot \cdot,
 $$
where  the extended maps are still denoted by $ \iota_0, q_0 $ and $ j $, respectively.
These maps induce the maps $ \varphi_1, \varphi_2,\varphi_3 $ and $\varphi_4$ in the following sequence
\begin{eqnarray}\label{11}
 \cdot \cdot \cdot \rightarrow\mathcal{H}^n_{app}(\mathcal{A},Y)\xrightarrow{\varphi_1} \mathcal{H}^n_{app}(\mathcal{A},X)\xrightarrow{\varphi_2} \mathcal{H}^n_{app}(\mathcal{A},X/Y) \xrightarrow{\varphi_3} \mathcal{H}_{app}^{n+1}(\mathcal{A},Y)\xrightarrow{\varphi_4} \cdot \cdot \cdot,
\end{eqnarray}
which are defined by
	$$\varphi_1(T_1+	\overline{\mathcal{B}^{n}(\mathcal{A},Y)}^{\text{ strong}})=\iota_0T_1+\overline{\mathcal{B}^{n}(\mathcal{A},X)}^{\text{ strong}},  $$
		$$\varphi_2(T_2+	\overline{\mathcal{B}^{n}(\mathcal{A},X)}^{\text{ strong}} )=q_0T_2+\overline{\mathcal{B}^{n}(\mathcal{A},X/Y)}^{\text{ strong}},  $$
		$$\varphi_3(T_3+	\overline{\mathcal{B}^{n}(\mathcal{A},X/Y)}^{\text{ strong}} )=jT_3+\overline{\mathcal{B}^{n+1}(\mathcal{A},Y)}^{\text{ strong}},  $$
			$$\varphi_4(T_4+	\overline{\mathcal{B}^{n+1}(\mathcal{A},Y)}^{\text{ strong}})=\iota_1T_4+\overline{\mathcal{B}^{n+1}(\mathcal{A},X)}^{\text{ strong}}.  $$

We show that the sequence 	(\ref{11}) is exact. To show exactness at
	$ \mathcal{H}^n_{app}(\mathcal{A},X) $,
	we have to show that $ \text{ran}\, \varphi_1=\ker \varphi_2 $, that is,
		for $ T\in \mathcal{Z}^{n}(\mathcal{A},X)$, the following cases are equivalent
	\begin{enumerate}
		\item[(i)] $ q_0T=\lim _\alpha \delta^{n-1}_{X/Y}T_\alpha $, for a net $ (T_\alpha)\subseteq\mathcal{B}^{n-1}(\mathcal{A},X/Y). $
		\item[(ii)] There exist $ S\in\mathcal{Z}^{n}(\mathcal{A},Y) $ and $ (S_\alpha)_\alpha \subseteq \mathcal{B}^{n-1}(\mathcal{A},X) $ such that $ T=\iota_0S+\lim _\alpha \delta^{n-1}_{X}S_\alpha.  $
	\end{enumerate}	
 For each $ \alpha $, we take $ S_\alpha  =pT_\alpha $ and we put
	$ S=T-\lim _\alpha \delta^{n-1}_{X}pT_\alpha  $ which shows that (i)$ \Rightarrow $(ii). For the converse,
 for every $ \alpha $, we put $ T_\alpha = qS_\alpha\in \mathcal{BL}^{n-1}(\mathcal{A},X/Y) $. Then
	$$qT=q\iota_0S+\lim _\alpha q\delta^{n-1}_{X}S_\alpha=\lim _\alpha \delta^{n-1}_{X/Y}T_\alpha,  $$  thus the implication (ii)$ \Rightarrow $(i) holds.
	
To show exactness at
$ \mathcal{H}^n_{app}(\mathcal{A},X/Y) $,
we need to verify that $ \text{ran} \varphi_2=\ker \varphi_3 $, that is,
for every $ T\in \mathcal{Z}^{n}(\mathcal{A},X/Y)$, the following cases are equivalent
\begin{enumerate}
	\item[(i)] There exists a net $ (S_\alpha)_\alpha \subseteq \mathcal{B}^{n}(\mathcal{A},Y) $ such that
	 $ jT=\lim _\alpha \delta^{n}_{Y}S_\alpha $.
	\item[(ii)] There exist $ S\in\mathcal{Z}^{n}(\mathcal{A},Y) $ and $ (R_\alpha)_\alpha \subseteq \mathcal{B}^{n-1}(\mathcal{A},X/Y) $ such that $ T=qS+\lim _\alpha \delta^{n-1}_{X/Y}R_\alpha.  $
\end{enumerate}		
 Taking $ S=pT-\lim _\alpha \delta^{n}_{Y}S_\alpha $ and $ R_\alpha=0  $	for each $ \alpha $, we have
$$T=qpT=qS+\lim _qjT=qS,  $$ so we see that (i)$ \Rightarrow $(ii).\\
To see the case
(ii)$ \Rightarrow $(i), put $S_\alpha=pqS-S-(\delta^{n-1}_{X}p-p\delta^{n-1}_{X/Y}   )R_\alpha\in \mathcal{BL}^{n}(\mathcal{A},Y)  $. Then	
	$$\delta^{n}_{X}S_\alpha=\delta^{n}_{X}pqS+\delta^{n}_{X}p\delta^{n-1}_{X/Y} R_\alpha. $$
	
So, by (ii) we have	$\delta^{n}_{X}pqS=\delta^{n}_{X}pT- \lim _\alpha  \delta^{n}_{X}p\delta^{n-1}_{X/Y} R_\alpha.   $ This yields $ \lim _\alpha  \delta^{n}_{X} S_\alpha=\delta^{n}_{X}pT=jT. $
	
To show exactness at
$ \mathcal{H}^{n+1}_{app}(\mathcal{A},Y) $,
we need to show that $ \text{ran}\, \varphi_3=\ker \varphi_4 $, that is,
 for every $ T\in \mathcal{Z}^{n+1}(\mathcal{A},Y)$, the following cases are equivalent
\begin{enumerate}
	\item[(i)] There exists a net $ (S_\alpha)_\alpha \subseteq \mathcal{BL}^{n}(\mathcal{A},Y) $ such that
	$ \iota_1T=\lim _\alpha \delta^{n}_{Y}S_\alpha $.
	\item[(ii)] There exist $ S\in\mathcal{Z}^{n}(\mathcal{A},X/Y) $ and $ (R_\alpha)_\alpha \subseteq \mathcal{BL}^{n}(\mathcal{A},Y) $ such that $ T=jS+\lim _\alpha \delta^{n}_{Y}R_\alpha=\delta^{n}_{X}pS+ \lim _\alpha \delta^{n}_{Y}R_\alpha.  $
\end{enumerate}		
By putting $ S=0 $ and $ R_\alpha=S_\alpha, $ we have
$$ T=\iota_1 T= \lim _\alpha \delta^{n}_{Y}S_\alpha=\lim _\alpha \delta^{n}_{Y}R_\alpha.$$ Thus (i)$ \Rightarrow $(ii).\\
(ii)$ \Rightarrow $(i) follows from
$$ \lim _\alpha \delta^{n}_{Y}S_\alpha=\delta^{n}_{X}pS+\lim _\alpha \delta^{n}_{Y}R_\alpha=\iota T,$$
where $S_\alpha=pS+ R_\alpha.  $
\end{proof}
\begin{cor}
Let $ \mathcal{A} $ be a Banach algebra with a bounded approximate identity and let $ X $ be  a Banach $ \mathcal{A} $-bimodule. Then $ X_1=\{ a\cdot x\cdot b: a, b\in \mathcal{A}, x \in X\} $ is a closed neo-unital submodule of $ X $ and $ X_1^\perp $ is complemented in $ X^* $ and for every $ n $, $\mathcal{H}^{n}_{app}(\mathcal{A},X^*)= \mathcal{H}^{n}_{app}(\mathcal{A},X_1^*).  $
\end{cor}
\begin{proof}
Let $ X_2=\{a\cdot x: a \in \mathcal{A},x \in X  \}. $	By \cite[Proposition 1.8]{J},
	$ X_1 $ and $ X_2 $ are closed submodule of $ X $ and $ X_1^\perp  $ is complemented in $ X^* $. Also $ X_1 $ is a closed submodule of $ X_2 $. By Theorem
\ref{long} the sequence $$ \cdot \cdot \cdot \rightarrow \mathcal{H}^{n}_{app}(\mathcal{A},X_2^\perp)\rightarrow \mathcal{H}^{n}_{app}(\mathcal{A},X^*) \rightarrow \mathcal{H}^{n}_{app}(\mathcal{A},X_2^*)\rightarrow  \cdot \cdot \cdot $$	
	is exact. Since $ \mathcal{A} $ acts trivially from left of $ X/X_2 $, by Remark \ref{triv}, $$ \mathcal{H}^{n}_{app}(\mathcal{A},X_2^\perp)=\mathcal{H}^{n}_{app}(\mathcal{A},(X/X_2)^*)=\{0\}.$$
	Similarly $ \mathcal{H}^{n+1}_{app}(\mathcal{A},X_2^\perp)=\{0\} $.
		Thus we obtain an exact sequence $$ \{0\}\rightarrow \mathcal{H}^{n}_{app}(\mathcal{A},X^*) \rightarrow \mathcal{H}^{n}_{app}(\mathcal{A},X_2^*)\rightarrow \{0\}, $$ which means that
	$ \mathcal{H}^{n}_{app}(\mathcal{A},X^*) = \mathcal{H}^{n}_{app}(\mathcal{A},X_2^*) $.
	Since $ X_1 $ is a closed submodule of $ X_2 $,
	 replacing $ X $ with $ X_2 $ and $ X_2 $ with $ X_1 $,   we conclude that $\mathcal{H}^{n}_{app}(\mathcal{A},X_2^*)=\mathcal{H}^{n}_{app}(\mathcal{A},X_1^*).$
\end{proof}
\begin{rem}
	Let $ \mathcal{A} $ be a Banach algebra and let $ X $ be  a Banach $ \mathcal{A} $-bimodule. Then, for every $ n \in \mathbb{N} $,
	$ \mathcal{BL}^n(\mathcal{A},X) $ is a Banach $ \mathcal{A} $-bimodule via the following actions,
	$$( a\cdot T)(a_1,..., a_{n})=  a\cdot T(a_1,..., a_{n}), $$ and
	\begin{align*}
		(T\cdot a)(a_1,..., a_{n})
		&:= T(aa_1,..., a_{n})\\
		&+\sum_{k=1}^{n-1}(-1)^kT(a,a_1,..., a_ka_{k+1},..., a_{n})\\
		&+(-1)^{n}T(a,a_1,..., a_{n-1})\cdot a_{n}, \quad (T\in \mathcal{BL}^n(\mathcal{A},X), a_1,..., a_{n}\in \mathcal{A}).
	\end{align*}
\end{rem}
In the following theorem we will give an analogue of reduction of dimension for approximate cohomology.
\begin{thm}\label{red}
Let $ \mathcal{A} $ be a Banach algebra and let $X $ be  a Banach $ \mathcal{A} $-bimodule. Then for every $ n\in \mathbb{N} $, $$\mathcal{H}^{n+1}_{app}(\mathcal{A},X) \cong \mathcal{H}^{1}_{app}(\mathcal{A},\mathcal{BL}^n(\mathcal{A},X)). $$
\end{thm}
\begin{proof}
	For every $ n\in \mathbb{N} $, consider the coboundary operator
 $$\delta_n^1:\mathcal{BL}(\mathcal{A},\mathcal{BL}^n(\mathcal{A},X))\rightarrow  \mathcal{BL}^2(\mathcal{A},\mathcal{BL}^n(\mathcal{A},X)). $$

 By \cite[Lemma 2.4.5]{R} for $ i=1,2 $,
 we have an  isometric isomorphism
$$ \zeta^i:\mathcal{BL}^{n+i}(\mathcal{A},X)\rightarrow \mathcal{BL}^i(\mathcal{A},\mathcal{BL}^n(\mathcal{A},X)), $$
defined by $$ \zeta^iT(a_1,..., a_i)(a_{i+1},...,a_{i+n})=T(a_1,..., a_i,a_{i+1},...,a_{i+n}), \quad (a_1,...,a_{i+n}\in\mathcal{A},   T\in \mathcal{BL}^{n+i}(\mathcal{A},X)). $$
Then for every $ a, b \in \mathcal{A} $,
\begin{align*}
	\delta_n^1 \zeta^1(T)(a,b)(a_1,...,a_n)
	&=a\cdot (\zeta^1T)(b)(a_1,...,a_n)
	- (\zeta^1T)(ab)(a_1,...,a_n)
	+(\zeta^1T)(a)\cdot b(a_1,...,a_n) \\
&=a\cdot T(b,a_1,...,a_n)
	- T(ab,a_1,...,a_n)
	+T(a,ba_1,...,a_n)+... \\
   &\qquad\qquad\qquad\qquad\qquad\qquad\qquad +(-1)^nT(a,b,a_1,...,a_{n-1})\cdot a_n \\
    &=\delta^{n+1}	T(a,b,a_1,...,a_n)\\
    &=\zeta^2 \delta^{n+1}(T)(a,b)(a_1,...,a_n),
\end{align*}
so we have
\begin{eqnarray}
\delta_n^1\circ \zeta^1=\zeta^2\circ \delta^{n+1},
\end{eqnarray}
that is, the following diagram  commutes.
\begin{center}
	\begin{tikzpicture}
	\matrix [matrix of math nodes,row sep=1cm,column sep=1cm,minimum width=1cm]
	{
		|(A)| \displaystyle \mathcal{BL}^{n+1}(\mathcal{A},X)  &   |(B)|  \mathcal{BL}^{n+2}(\mathcal{A},X)    \\
		|(C)|   \mathcal{BL}(\mathcal{A},\mathcal{BL}^n(\mathcal{A},X))    &   |(D)|  \mathcal{BL}^{2}(\mathcal{A},\mathcal{BL}^n(\mathcal{A},X)). \\
	};
	\draw[->]    (A)-- node [above] { $ \delta^{n+1} $}(B);
	\draw[->]   (A)--  node [left] { $ \zeta^1 $} (C);
	\draw[->]  (C)-- node [below]   { $\delta_{n}^1$}(D);
	\draw[->]  (B)-- node [right]  {$ \zeta^2 $}(D);
	\end{tikzpicture}
\end{center}
 We show that
 \begin{eqnarray}\label{z}
 	\mathcal{Z}^{n+1} (\mathcal{A},X) \cong \mathcal{Z}^{1} (\mathcal{A},\mathcal{BL}^n(\mathcal{A},X)).
 \end{eqnarray}
$$
	T\in\mathcal{Z}^{n+1} (\mathcal{A},X)\Leftrightarrow  \delta^{n+1}T=0 \Leftrightarrow \delta_n^1 \zeta^1T=\zeta^2 \delta^{n+1}T=0 \Leftrightarrow \zeta^1T\in \mathcal{Z}^{1} (\mathcal{A},\mathcal{BL}^n(\mathcal{A},X)).
$$
Now we show that $ \overline{\mathcal{B}^{n+1}(\mathcal{A},X)}^{\text{ strong}}\cong \overline{\mathcal{B}^{1}(\mathcal{A}, \mathcal{BL}^n(\mathcal{A},X))}^{\text{ strong}} $.
The following diagram commutes
\begin{center}
	\begin{tikzpicture}
	\matrix [matrix of math nodes,row sep=1cm,column sep=1cm,minimum width=1cm]
	{
		|(A)| \displaystyle \mathcal{BL}^{n}(\mathcal{A},X)  &   |(B)|  \mathcal{BL}^{n+1}(\mathcal{A},X)    \\
		|(C)|   \mathcal{BL}^{n}(\mathcal{A},X)    &   |(D)|  \mathcal{BL}(\mathcal{A},\mathcal{BL}^n(\mathcal{A},X)), \\
	};
	\draw[->]    (A)-- node [above] { $ \delta^{n} $}(B);
	\draw[->]   (A)--  node [left] { $ id $} (C);
	\draw[->]  (C)-- node [below]   { $\delta_{n}^0$}(D);
	\draw[->]  (B)-- node [right]  {$ \zeta^1 $}(D);
	\end{tikzpicture}
\end{center}
because
 \begin{align*}
\zeta^1 \delta^{n}(T)(a)(a_1,...,a_n)
	&=\delta^{n}(T)(a)(a_1,...,a_n)         \\
	&=a\cdot T(a_1,...,a_n)          \\
	&- T(aa_1,...,a_n)+...           \\
	&+(-1)^nT(a,a_1,...,a_{n-1})\cdot a_n \\
	&=(a\cdot T)(a_1,...,a_n)-(T\cdot a)(a_1,...,a_n)\\
	&=\delta_{n}^0(T)(a)(a_1,...,a_n),
\end{align*}
 that is, $ \zeta^1 \circ\delta^{n}=\delta_{n}^0  $.
Now if $ T\in \overline{\mathcal{B}^{n+1}(\mathcal{A},X)}^{\text{ strong}} $, then there exists a net $ (S_\alpha)_\alpha\subseteq \mathcal{BL}^{n}(\mathcal{A},X) $ such that $ T=\lim _\alpha \delta^{n}S_\alpha. $ Since $ \zeta^1 $ is an isometric isomorphism, we have $$ \zeta^1 T=\lim _\alpha \zeta^1\delta^{n}S_\alpha=\lim _\alpha \delta^{0}_nS_\alpha .$$
 Hence $\zeta^1 T\in \overline{\mathcal{B}^{1}(\mathcal{A}, \mathcal{BL}^n(\mathcal{A},X))}^{\text{ strong}}. $

Conversely, let $ T\in \mathcal{BL}^{n+1}(\mathcal{A},X) $ be such that $ \zeta^1 T\in \overline{\mathcal{B}^{1}(\mathcal{A}, \mathcal{BL}^n(\mathcal{A},X))}^{\text{ strong}} $.
Thus there exists a net  $ (R_i)_i\subseteq \mathcal{BL}^{n}(\mathcal{A},X) $ such that
$\zeta^1 T =\lim _i \delta^{0}_nR_i=\lim _i\zeta^1\delta^{n}R_i.   $
Again, since $ \zeta^1 $ is an isometric isomorphism, we have $T= \lim _i\delta^{n}R_i $, that is, $ T\in \mathcal{B}^{n+1}(\mathcal{A},X).$
Therefore \begin{eqnarray}\label{b}
 \overline{\mathcal{B}^{n+1}(\mathcal{A},X)}^{\text{ strong}}\cong \overline{\mathcal{B}^{1}(\mathcal{A}, \mathcal{BL}^n(\mathcal{A},X))}^{\text{ strong}}.
\end{eqnarray}
Now by equations (\ref{z}) and (\ref{b}) we obtain $ \mathcal{H}^{n+1}_{app}(\mathcal{A},X) \cong \mathcal{H}^{1}_{app}(\mathcal{A},\mathcal{BL}^n(\mathcal{A},X)) $ as required.
\end{proof}

The conclusion before \cite[Proposition 5.1]{J} showes that a Banach algebra  $\mathcal{A}$ is amenable if and only if  $\mathcal{H}^n(\mathcal{A},X^*)=\{ 0\} $ for every Banach $\mathcal{A}$-bimodule $X$ and  for every $n\geq 1$.  For more detail, we refer the reader to  \cite[Theorem 2.4.7]{R}. The following shows this remains true for approximately amenable.
\begin{thm}
Let $ \mathcal{A} $ be a Banach algebra. Then the following are equivalent:
	\begin{enumerate}
		\item[(i)] $ \mathcal{A} $ is  approximate amenable.
		\item[(ii)] $\mathcal{H}^n_{app}(\mathcal{A},X^*)=\{ 0\} $ for every Banach $\mathcal{A}$-bimodule $X$ and  for every $n\geq 1$.
	\end{enumerate}		
\end{thm}
\begin{proof} (ii)$\Rightarrow$(i) is clear. We only prove  assertion (i)$\Rightarrow$(ii).
	
Let $Y= \underbrace{\mathcal{A}\hat{\otimes} \mathcal{A}\hat{\otimes}\cdots\hat{\otimes}\mathcal{A}}_{(n-1)-times}\hat{\otimes}X$. 
Then $ Y^*\simeq {\mathcal{BL}^{n-1}(\mathcal{A},X^*)} $, so by assumption and by Theorem \ref{red} we have
$$\mathcal{H}^n_{app}(\mathcal{A},X^*)=\mathcal{H}^1_{app}(\mathcal{A},Y^*)=\{ 0\}.  $$
\end{proof}
\begin{pro}
Let $\mathcal{A}  $ be an approximately amenable Banach algebra and let $X$ be a commutative Banach 	$ \mathcal{A} $-bimodule. Then, $\mathcal{H}^{1}_{app}(\mathcal{A},X)=\mathcal{H}^{2}_{app}(\mathcal{A},X).$
\end{pro}
\begin{proof}
Let $ D\in \mathcal{Z}^{1}(\mathcal{A},X)\subseteq \mathcal{Z}^{1}(\mathcal{A},X^{**}) $. Since 	$\mathcal{A}  $ is   approximately amenable, $ D \in \overline{\mathcal{B}^{1}(\mathcal{A},X^{**})}^{\text{ strong}}=\{0\} $, since $ X^{**} $ is a commutative Banach 	$ \mathcal{A} $-bimodule. Thus $ D=0 $, so $ \mathcal{H}^{1}_{app}(\mathcal{A},X)=\{0\}. $
	
	Next, we show that $ \mathcal{H}^{2}_{app}(\mathcal{A},X)=\{0\}. $
	By Theorem \ref{red} and by approximate amenability of $ \mathcal{A} $,
	\begin{align*}
		\mathcal{H}^{2}_{app}(\mathcal{A},X^{**})
		&=\mathcal{H}^{1}_{app}(\mathcal{A},\mathcal{BL}(\mathcal{A},X^{**}))\\
		&=\mathcal{H}^{1}_{app}(\mathcal{A},(\mathcal{A} \hat{\otimes}X^*)^*)\\
		&=\{0\}.
	\end{align*}
Now let $  T\in \mathcal{Z}^{2}(\mathcal{A},X)\subseteq \mathcal{Z}^{2}(\mathcal{A},X^{**})=\overline{\mathcal{B}^{2}(\mathcal{A},X^{**})}^{\text{ strong}} $.	
There exists a net $ (R_\alpha)_\alpha\subseteq \mathcal{B}^{2}(\mathcal{A},X^{**})$ such that $ R_\alpha \xrightarrow{\text{ strongly}} T.$  Let $ q:X^{**}\rightarrow X^{**}/X $ be the usual quotient map, so $ qR_\alpha\in \mathcal{B}^{2}(\mathcal{A},X^{**}/X) $. 	
Hence $ \delta qR_\alpha=q\delta R_\alpha=0 $, that is, $ qR_\alpha\in \mathcal{Z}^{1}(\mathcal{A},X^{**}/X)=\{0\} $, by the first part. Thus $ qR_\alpha=0, $
	so $ R_\alpha \in \mathcal{B}^{2}(\mathcal{A},X). $ Therefore $ T\in \overline{\mathcal{B}^{2}(\mathcal{A},X)}^{\text{ strong}} $ and the proof is complete.
\end{proof}
\begin{pro}
	Let $H$ be a subgroup of a discrete group $G$. If $ \mathcal{H}^{2}_{app}(\ell^1(G),X^*)= \{0\}$, for all Banach $ \ell^1(G) $-bimodules $X$, then $ \mathcal{H}^{2}_{app}(\ell^1(H),X^*)= \{0\}$ for all  Banach $ \ell^1(H) $-bimodules $X$.
\end{pro}
\begin{proof}
	Let $X$ be a Banach $ \ell^1(H) $-bimodule. By a conclusion similar to the prior of
	\cite[Theorem 2.5]{J}, we may assume that $X$ is only a right $H$-module. Let $M$ be a set of representatives of the right cosets of $H$ in $G$, that is, $ M\subseteq G $ and for all $ g \in G $, $ Hg\cap M $ contains exactly one element denoted by $ m(Hg) $. The set $M$ can be chosen with $ m(H)=e, $ the identity element.
	
	Define a map $ h:G\rightarrow H $ satisfying $ h(g)m(Hg)=g. $ Let $ Y=\ell^1(G/H)\hat{\otimes} X$ which we isometrically identify with $ \ell^1(G/H, X)$, by \cite[Theorem B.2.12]{R}. We define a right $G$-action on $Y$ by
	$$ (f\cdot g_1)(Hg_2)=f(Hg_2g_1^{-1})\cdot h(m(Hg_2g_1^{-1})g_1), $$
	where  $ f \in \ell^1(G/H, X) $.
	We only check the associativity of the action.
	Since $ m(Hm(Hg_1g_2^{-1})g_2)=m(Hg_1) $ for all $ g_1 $ and $ g_2 $, by the definition of $ h $
	\begin{eqnarray}
	h(m(Hg_3g_2^{-1}g_1^{-1})g_1)m(Hg_3g_2^{-1})g_2=m(Hg_3(g_1g_2)^{-1})g_1g_2,
	\end{eqnarray}
which yields $$ 	h(m(Hg_3g_2^{-1}g_1^{-1})g_1)h(m(Hg_3g_2^{-1})g_2)m(Hm(Hg_3(g_1g_2)^{-1})g_1g_2)=m(Hg_3(g_1g_2)^{-1})g_1g_2. $$	
	Hence
	\begin{align*}
	[(f\cdot g_1)\cdot g_2](Hg_3)
	&=f(Hg_3g_2^{-1}g_1^{-1})\cdot h(m(Hg_3g_2^{-1}g_1^{-1})g_1)\cdot h(m(Hg_3g_2^{-1})g_2)\\
	&= f(Hg_3g_2^{-1}g_1^{-1})\cdot h(m(Hg_3(g_1g_2)^{-1})g_1g_2)\\
	&=(f\cdot (g_1g_2))(Hg_3).
	\end{align*}
	The dual space of $Y$ is $ \ell^\infty(G/H,X^*) $ with the dual action
	$$(g_1\cdot F)(Hg_2)=h(m(hg_2)g_1)\cdot F(Hg_2g_1) \quad (F\in  \ell^\infty(G/H,X^*)).$$
	
If $ T\in \mathcal{Z}^{2}(H,X^*),  $ then we define $ \tilde{T}\in \ell^\infty(G^2,X^*)$ by
$$ \tilde{T}(g_1,g_2)(Hg_0)=T(h(m(Hg_0)g_1),h(m(Hg_0)g_1^{-1})h(m(Hg_0)g_1g_2).$$	
	If $ g_i\in G, (i=1,...,4) $ then substituting $ g_5=h(m(Hg_0)g_1), g_6=g_5^{-1}h(m(Hg_0)g_1g_2) $ and\\
	$ g_7=(g_5g_6)^{-1}h(m(Hg_0)g_1g_2g_3) $ in $ \delta T(g_5,g_6,g_7)=0, $
	we have $ \delta \tilde{T}(g_5,g_6,g_7)(Hg_0)=0 $, so that $ \tilde{T}\in \mathcal{Z}^{2}(G,X^*)$. Thus by assumption, there exists a net $ (S_\alpha)_\alpha\subseteq \ell^1(G, X^*) $ with $$\tilde{T}(h_1,h_2)(H)=\lim _\alpha \delta S_\alpha (h_1,h_2)(H). $$
	For each $ \alpha $ we define $ R_\alpha \in \ell^1(H, X^*) $ by
	$ R_\alpha (h_0)=S_\alpha (h_0)(H) $. Then $ T=\lim _\alpha \delta R_\alpha $, which means that $ T \in \overline{\mathcal{B}^{2}(H,X^*)}^{\text{ strong}} $.
	Thus $ \mathcal{H}^{2}_{app}(\ell^1(H),X^*)= \{0\}$.
\end{proof}

\end{document}